\documentclass[12pt]{amsart}
\usepackage{amssymb}
\usepackage{amsmath}
\usepackage{bbm}
\usepackage{color}
\usepackage{verbatim}

\definecolor{darkred}{rgb}{0.75,0,0.3}

\newcommand\AND{\quad\text{and}\quad}
\newcommand\as{a}
\newcommand\Aut{\operatorname{\sf Aut}}
\newcommand\bigast{\text{\Large $\ast$}}
\DeclareMathOperator*{\freeprod}{\bigast}
\newcommand\C{\mathbb C}
\newcommand\de{\delta}
\newcommand\Ex{\mathsf{E}}
\newcommand\hor{\mathfrak{h}}

\newcommand\la{\lambda}
\newcommand\mm{\mathsf m}
\newcommand\N{\mathbb N}
\newcommand\Prob{\mathsf{Pr}}
\newcommand\R{\mathbb R}
\newcommand\res{\mathsf{res}}
\newcommand\spec{\mathsf{spec}}
\newcommand\sr{m}
\newcommand\Ss{\N}

\newcommand\T{\mathbb T}
\newcommand\uno{\mathbf{1}}
\newcommand\wh{\widehat}

\newcommand{\massimo}[1]{{\leavevmode\color{red}{#1}}}
\newcommand\dist{\mathrm{dist}}

\numberwithin{equation}{section}

\newtheoremstyle{mythm}
  {9pt}
  {9pt}
  {\itshape}
  {0pt}
  {\bfseries}
  {}
  { }
  {\thmnumber{(#2)}\thmname{ #1}\thmnote{ #3}}

\newtheoremstyle{mydef}
  {9pt}
  {9pt}
  {\normalfont}
  {0pt}
  {\bfseries}
  {}
  { }
  {\thmnumber{(#2)}\thmname{ #1}\thmnote{ #3}}

\theoremstyle{mythm}
\newtheorem{thm}[equation]{Theorem.}
\newtheorem{pro}[equation]{Proposition.}
\newtheorem{lem}[equation]{Lemma.}
\newtheorem{cor}[equation]{Corollary.}

\theoremstyle{mydef}

\newtheorem{rmk}[equation]{Remark.}
\newtheorem{rmks}[equation]{Remarks.}
\newtheorem{rmksan}[equation]{Analytic continuation.}
\newtheorem{rmksig}[equation]{Remarks on $\sigma$-additivity.}
\newtheorem{rmkgp}[equation]{General transitive group actions.}

\begin{document}$\,$ \vspace{-1truecm}
\title{Boundary representations of $\la$-harmonic and polyharmonic functions on trees}
\author{\bf Massimo A. Picardello, Wolfgang Woess}
\address{\parbox{.8\linewidth}{Dipartimento di Matematica\\ 
Universit\`a di Roma ``Tor Vergata''\\
I-00133 Rome, Italy}}
\email{picard@axp.mat.uniroma2.it}

\address{\parbox{.8\linewidth}{Institut f\"ur Diskrete Mathematik,\\ 
Technische Universit\"at Graz,\\
Steyrergasse 30, A-8010 Graz, Austria}}
\email{woess@tugraz.at}

\thanks{Supported by Austrian Science Fund projects FWF P24028 and W1230. 
The second author acknowledges support by as a distinguished visiting scientist
at TU Graz}
\begin{abstract}
On a countable tree $T$, allowing vertices with infinite degree,
we consider an arbitrary stochastic irreducible nearest neighbour transition
operator $P$. We provide a boundary integral representation for
general eigenfunctions of $P$ with eigenvalue $\lambda \in C$.
This is possible whenever $\lambda$ is in the resolvent set of $P$
as a self-adjoint operator on a suitable $\ell^2$-space and the 
on-diagonal elements of the resolvent (``Green function'') do not vanish
at $\lambda$. We show that when $P$ is invariant under a transitive
(not necessarily fixed-point-free) group action, the latter condition
holds for all $\la \ne 0$ in the resolvent set. These results 
extend and complete previous results by Cartier, by Fig\`a-Talamanca and Steger,
and by Woess. For those eigenvalues, we also provide an integral
representation of $\lambda$-polyharmonic functions of any order $n$, 
that is, functions $f: T \to \C$ for which $(\la \cdot I - P)^nf=0$. 
This is a far-reaching extension of work of Cohen et al., who provided 
such a representation for simple random walk on a homogeneous tree and 
eigenvalue $\la =1$. Finally, we explain the (much simpler) analogous
results for ``forward only'' transition operators, sometimes also called
martingales on trees.
\end{abstract}

\maketitle
       
\markboth{{\sf M. A. Picardello and W. Woess}}
{{\sf Boundary representations on trees}}
\baselineskip 15pt

\section{Introduction}\label{sec:intro}

Let $T$ be a countable tree without leaves (vertices with only one neighbour).
We allow vertices with countably many neighbours. On $T$, we consider a
random walk which is of nearest neighbour type, that is, the transition probabilities
$p(x,y)$ are $> 0$ if and only if $x$ and $y$ are neighbours.
We are interested in general eigenfunctions of the transition operator 
acting on functions $f: T \to \C$ by
$$
Pf(x) = \sum_y p(x,y)f(y)\,.
$$
When this is an infinite sum, it is required to converge absolutely.
For $\la \in \C$, a \emph{$\la$-harmonic function} $h: T \to \C$ is one that
satisfies $Ph = \la \cdot h$. If we consider $\Delta = P - I$ as a 
discrete Laplace operator (where $I$ is the identity operator), then $h$ is an
eigenfunction of the Laplacian with eigenvalue $\la -1$.

First of all, if $\la$ is real and $\la > \rho(P)\,$, the \emph{spectral radius}
of the random walk, then every positive $\la$-harmonic function has
a unique integral representation over the boundary at infinity of $T$ with respect
to the \emph{$\la$-Martin kernel}, in analogy with the Poisson representation
of positive eigenfunctions of the Laplacian on the unit disk. The same holds for
$\la = \rho(P)$ in case $P$ is \emph{$\rho$-transient.}

Furthermore, for those eigenvalues one has a similar integral representation for 
\emph{any} real (or complex) eigenfunction, where the integral of the Martin kernel
is taken with respect to a \emph{distribution}, that is, a finitely additive signed 
measure which is defined on the collection of all \emph{boundary arcs.}
However, this distribution does in general not extend to a $\sigma$-additive
measure on the Borel $\sigma$-algebra of the boundary. 

When $T$ is locally finite and $\la > \rho(P)$ (or $=\rho(P)$ in the $\rho$-transient 
case), this is comprised in the seminal paper of {\sc Cartier}~\cite{Ca}; see 
also {\sc Picardello and Woess}~\cite{PiWo}.
The analogous result in the non-locally finite case is comprised in 
the textbook of {\sc Woess}~\cite[\S 9.D]{W-Markov}, which seemingly has remained
unobserved by researchers working in this field.

The first goal of this paper is to formulate and prove this general representation
theorem in the greatest possible generality, that is, for arbitrary complex-valued 
$\la$-harmonic functions, where $\la \in \res(P) = \C \setminus \spec(P)$, the resolvent set
of $P$.
Here, $P$ is interpreted as a self-adjoint, bounded operator on $\ell^2(T,\mm)$, where $\mm$ is
the measure on the vertex set which makes $P$ reversible, that is, $\mm(x)p(x,y) = \mm(y)p(y,x)$
for all $x,y \in T$. That measure is unique up to normalisation.
There is one restriction for the general representation, namely, that
$\la$ has to be such that the on-diagonal matrix elements $G(x,x|\la)$ of the resolvent 
operator $\mathfrak{G}(\la) = (\la\cdot I - P)^{-1}$ do not vanish. This holds always
when $|\la| \ge \rho(P)$.
(The use of the letter $\mathfrak{G}$ is motivated by the usual name ``Green function''
for its matrix elements.) We show that the condition on $\la$ allows us
to construct the general analogue of the $\la$-Martin kernel.
Our corresponding integral representation is Theorem \ref{thm:general} below.
Among other, this generalises a similar result of {\sc Fig\`a-Talamanca and Steger}~\cite{FiSt} 
concerning the
case when $T$ is the regular, locally finite tree which is the Cayley graph of the
group $\langle a_1\,,\dots, a_r \mid a_j= a_j^{-1} \,, j=1,\dots, r\rangle$, and when
$P$ is invariant under that group. 

Our Theorem \ref{thm:general} does not require any group structure. But in addition, 
we also study in more detail the specific case where $T$ is not necessarily locally finite
and $P$ is invariant under the action of an arbitrary group of automorphisms of $T$ 
which is not required to act with trivial vertex stabilisers, so that $T$ is
not necessarily a Cayley graph of that group. Indeed, its closure will be a locally compact group
that may be non-discrete and even non-unimodular. In this very general group-invariant situation, 
we  provide an extension of another result of \cite{FiSt}, namely, that the Green kernel
may vanish on the diagonal only for $\la = 0$, see Theorem \ref{thm:nonzero}. 
Thus, we have the integral representation in the group invariant case for all
$\la$-harmonic functions, where $\la \in \res(P)$, with the only possible
exception of $\la = 0$. 

Next, we turn our attention to \emph{$\la$-polyharmonic functions.} A function $f: T \to \C$
is $\la$-polyharmonic of order $n$, if 
$$
(\la\cdot I - P)^n f = 0.
$$  
If $n=1$, this means that $f$ is $\la$-harmonic. If $n=2$ this means that $\la\cdot f-Pf$
is $\la$-harmonic, and so on. 

In the setting of the classical Laplacian $\Delta$ on
a Euclidean domain, polyharmonic functions are functions for which $\Delta^n h \equiv 0$.
Their study goes back to work in the $19^{\text{th}}$ century, see e.g. 
{\sc Almansi}~\cite{Al}. A basic reference is the monograph by {\sc Aronszajn, 
Creese and Lipkin}~\cite{ACL}, and there is ongoing study. 
The discrete analogue of polyharmonic functions on trees ($\la = 1$)
was studied in a long paper by {\sc Cohen et al.}~\cite{CCGS}. For the special 
case of simple random
walk on a locally finite, homogeneous tree, they provide a boundary integral representation for
polyharmonic functions. Here, we provide a far-reaching generalisation: in Theorem \ref{thm:poly} 
we explain how this can by achieved more directly 
for arbitrary $\la$-polyharmonic functions in the general setting of a nearest neighbour 
random walk on a countable tree $T$ (locally finite or not), whenever 
$\la \in \res(P)$ fulfils the above  condition that $G(x,x|\la) \ne 0$ 
for every $x \in T$.  

Finally, in the appendix  \S \ref{sec:forward}, we explain how all the above results
are obtained very easily for ``forward only'' transition operators on rooted countable 
trees, see in particular Proposition \ref{pro:polyQ}. In this case, the only exceptional
eigenvalue is $\la =0$.

\section{Basic facts}\label{sec:basic}

We briefly recall the basic ingredients. For two vertices $x,y \in T$, we write $x \sim y$ 
if they are neighbours. The degree of $x$ is its number of neighbours. Given any pair 
of vertices $x, y$, the \emph{geodesic} or \emph{geodesic path} from $x$ to $y$ is the
unique shortest path $\pi(x,y)$ from $x$ to $y$, and the distance $d(x,y)$ is the
length (number of edges) of $\pi(x,y)$.
 
A \emph{ray} or \emph{geodesic ray} in $T$
is a sequence $[x_0\,,x_1\,,x_2\,,\dots]$ such that $x_{i-1}\sim x_i$ and all $x_i$
are distinct. Two rays are \emph{equivalent,} if they differ by finitely many initial
vertices. An \emph{end} of $T$ is an equivalence class of rays. For any vertex $x$ and end
$\xi$, there is a unique ray $\pi(x,\xi)$ which starts at $x$ and represents $\xi$.  We 
write $\partial T$ for the set of all ends of $T$. 
For $x, y \in T$ with $x \ne y$, the \emph{branch} or \emph{cone} $T_{x,y}$ is the subtree 
spanned by all vertices $w$ with $y \in \pi(x,w)$, and the \emph{boundary arc}
$\partial T_{x,y}$ is the set of all ends which have a representative ray in $T_{x,y}\,$.

We set $\wh T = T \cup \partial T$ and $\wh T_{x,y} = T_{x,y} \cup \partial T_{x,y}\,$.
We put the topology on $\wh T$ which is discrete on the vertex set, while a neighbourhood
base of $\xi \in \partial T$ is given by the collection of all $\wh T_{x,y}$ which contain
$\xi$. (Here, we may fix $x$ and vary only $y \ne x$.) We obtain a metrisable space which
is compact precisely when $T$ is locally finite (all vertex degrees are finite).
Otherwise, it is not complete. Following an idea of {\sc Soardi}, this defect can be 
overcome by introducing additional boundary points, one associated with each vertex of
infinite degree, see {\sc Cartwright, Soardi and Woess}~\cite{CaSoWo} or 
the exposition in \cite[\S 9.B]{W-Markov}.  
%
%

In order to describe convergence to ends, we choose a \emph{root} vertex $o \in T$.
For any pair of elements $v, w \in \wh T$, their \emph{confluent} $v \wedge w$ with respect
to $o$ is the last common vertex on the geodesics $\pi(o,v)$ and $\pi(o,w)$. Then
a sequence $(x_n)$ in $T$ converges to an end $\xi$ if and only if
$|x_n \wedge \xi| \to \infty$, where  $|x| = d(x,o)$. 

Next, let us turn to the random walk. Let $X_n$ be the random position at time $n \ge 0$.
The $n$-step transition probability $p^{(n)}(x,y) = \Prob[X_n=y | X_0=x]$ is the
$(x,y)$-element of the matrix power $P^n$, with $P^0=I$, the identity matrix. 
The \emph{spectral radius} 
$$
\rho = \rho(P) = \limsup_{n \to \infty} p^{(n)}(x,y)^{1/n}
$$ 
is independent of $x$ and $y$. Let $\la \in \C$ be such that $|\la| > \rho$. 
The associated \emph{Green function} is 
\begin{equation}\label{eq:green}
G(x,y|\la) = \sum_{n=0}^{\infty} p^{(n)}(x,y)\, \la^{-n-1}
\end{equation}
The series converges absolutely. It may or may not converge at $\la =\rho$, in which case we say
that the random walk is \emph{$\rho$-transient,} and \emph{$\rho$-recurrent,} respectively. 
Next, let us write 
$$
f^{(n)}(x,y) = \Prob[X_n = y,\; X_k \ne y \; \text{ for }\; k < n \mid X_0=x]
$$
for the probability that the random walk starting at $x$ reaches $y$ at time $n$, but not
before. We set
\begin{equation}\label{eq:FU}
F(x,y|\la) = \sum_{n=0}^{\infty} f^{(n)}(x,y)\, \la^{-n} \AND U(x,x|\la) = \sum_v p(x,v) F(v,x|\la)\,.
\end{equation}
Writing $U(x,x|\la)$ as a power series in $\la^{-1}$, the coefficient of $\la^{-n}$ is
the probability that the first return to the starting point $x$ occurs at time $n$ ($n \ge 1$).
These series certainly converge for $|\la| > \rho$, since this is true for the
Green function.

\begin{lem}\label{lem:crucial} The following identities hold on a tree for $|\la| > \rho$
as well as for $\la = \pm\rho$.
\\[4pt]
\emph{(a)} For any geodesic path $[x_0\,, x_1\,,\dots, x_k]\,$,
$$
F(x_0\,,x_k|\la) =  F(x_0\,,x_1|\la) F(x_1\,,x_2|\la) \cdots F(x_{k-1}\,,x_k|\la) 
$$
\emph{(b)} For any pair of neighbours $x,y$, we have $F(x,y|\la) \ne 0$ and
$$
\la \,F(x,y|\la) = p(x,y) + \sum_{v \ne y} p(x,v) \, F(v,x|\la)\, F(x,y|\la).
$$ 
\emph{(c)} For any pair of neighbours $x,y$, we have 
$$
\sum_v p(x,v)\,\big| F(v,x|\la) \big| \le |\la|
\AND \big| F(x,y|\la) F(y,x|\la) \big| \le 1,
$$ 
with equality only for $\la = \pm \rho$ in the $\rho$-recurrent case, and
$$
G(x,x|\la) = \frac{1}{\la - U(x,x|\la)} = \frac{F(x,y|\la)/p(x,y)}{1-F(x,y|\la) F(y,x|\la)}\,.
$$
\emph{(d)} For arbitrary $x, y \in T$,
$$
G(x,y|\la) =  F(x,y|\la)G(y,y|\la)\,.
$$
In case of vertices $x$ with infinite degree, the sums appearing in (b) and (c) converge
absolutely.
\end{lem}

\begin{proof}
For (a) and the formula displayed in (b), see e.g. \cite[Prop. 9.3]{W-Markov}, 
while the first of the two displayed formulas for $G(x,x|\la)$ in (c), 
as well as (d) are valid for general Markov chains \cite[Thm 1.38]{W-Markov}. 

Regarding (c), as mentioned above,
the function $U(x,x|\la)$ is a power series in $\la^{-1}$
whose coefficients are non-negative.
It  converges for $|\la| > \rho$, and for positive $\la \in [\rho\,,\,+\infty)$, it is
continuous and decreasing with value $0$ at $+\infty$. If one had $U(x,x|\la_0) = \la_0$ for some
positive $\la_0 > \rho$, then $G(x,x|\la)$ would have a pole at $\la_0$, in contradiction
with the fact that it is analytic for $|\la| > \rho$. Thus
$U(x,x|\la)  < \la$ for every real $\la > \rho$, and $U(x,x|\rho) = U(x,x|-\rho) \le \la$ with 
equality precisely 
when the random walk is $\rho$-recurrent.  Consequently, for complex $\la$ with $|\la| > \rho$,
and even for $|\la| \ge \rho$ in the $\rho$-transient case,
we have 
$$
\sum_v p(x,v)\,\big| F(v,x|\la) \big| \le 
U\bigl(x,x\big|\,|\la|\bigr)  < |\la|\,.
$$
This proves the first inequality stated in (c), and it shows that for $|\la|=\rho$,
also $\big|F(x,y|\la)\big| \le F(x,y|\rho) < \infty$.

It also implies that in the context of (b), one has for complex $|\la| > \rho$, 
resp., for $|\la|\ge \rho$ in the $\rho$-transient case that 
$$
 \sum_{v \ne y} p(x,v) \bigl| F(v,x|\la) \bigr| < |\la|\,,
$$ 
which in turn implies $F(x,y|\la) \ne 0\,$, as stated in (b).

These arguments also comprise the statement on absolute convergence.

Finally, for the second inequality and the formula for $G(x,x|\la)$ in (c), we use (b) to
see that 
\begin{equation}\label{eq:FFF}
F(x,y|\la) \Bigl(\la - U(x,x|\la)\Bigr) = p(x,y)\Bigl(1 - F(x,y|\la)F(y,x|\la)\Bigr)\,.
\end{equation}
This yields the displayed formula, as well as the second
inequality in the same way as the first one.
\end{proof}

At this point, 
we can define the \emph{$\la$-Martin kernel} by
\begin{equation}\label{eq:Mkernel}
 K(x,w|\la) = \frac{G(x,x \wedge w|\la)}{G(o,x \wedge w|\la)}
= \frac{F(x,x \wedge w|\la)}{F(o,x \wedge w|\la)}\,,\quad x \in T\,,\; w \in \wh T\,.
\end{equation}
For fixed $x$, it is continuous in the second variable.
It is an easy exercise to derive from Lemma \ref{lem:crucial} that for any $\xi \in \partial T$,
the function $x \mapsto K(x,\xi|\la)$ is a $\la$-harmonic function (see also below):
\begin{equation}\label{eq:eigen}
\sum_{y \sim x} p(x,y) K(y,\xi|\la) = \la \, K(x,\xi|\la) \quad \text{for every}\; x \in T. 
\end{equation}
Note that when $x$ has infinite degree, this sum does indeed converge absolutely by Lemma
\ref{lem:crucial}.
The following is well known.

\begin{pro}\label{pro:martin}
For real $\la > \rho$, as well as for $\la = \rho$ in the $\rho$-transient case,
 every \emph{positive} $\la$-harmonic function $h$ has a unique \emph{Poisson-Martin} integral
representation
$$
h(x) = \int_{\partial T} K(x,\xi|\la)\, d\nu^h(\xi)\,,
$$
where $\nu^h$ is a non-negative Borel measure on $\partial T$ with total mass $h(o)$.
\end{pro}

For the case when $T$ is locally finite, this as well as the identities of Lemma \ref{lem:crucial}
are comprised in the influential and elegant paper of {\sc Cartier}~\cite{Ca}. 
The extension to the non-locally finite case goes back to {\sc Soardi}, see \cite{CaSoWo} and 
\cite[\S 9.B+C]{W-Markov}. In the last two references, it is assumed that $\la=1$. For general
positive $\la$ and $h$, the result is obtained by applying the standard case to the 
\emph{$h$-process,} whose transition probabilities are 
$p_h(x,y) = p(x,y)h(y)\big/\bigl(\la\, h(x)\bigr)$.

As outlined in the introduction, one of the aims of this paper is an integral representation
for all $\la$-harmonic functions for arbitrary 
$\la \in \res(P) = \C \setminus \spec(P)$ up to some possible exceptional values.
Proposition \ref{pro:martin} will follow from that general result.
The spectrum $\spec(P) \subset [-\rho\,,\,\rho]$
is going to be defined in the next part, which is going to be important for
all subsequent considerations.

\medskip
 
\begin{rmksan}\label{rmk:an} 
We define a measure $\mm$ on $T$ as follows: 
$$
\text{for }\; x \in T\;{ with }\; \pi(o,x) = [x_0\,,x_1\,,\dots, x_k]\,,\quad
\mm(x)= \frac{p(x_0\,,x_1) \cdots p(x_{k-1}\,,x_k)}{p(x_1\,,x_0) \cdots p(x_k\,,x_{k-1})}.
$$
In particular, $\mm(o)=1$. We have \emph{reversibility:} $\mm(x)p(x,y) = \mm(y)p(y,x)$ for all $x,y$.
Thus, $P$ acts as a self-adjoint operator on the Hilbert space $\ell^2(T,\mm)$ of all functions 
$f:T \to \R$
with $\langle f, f \rangle < \infty$, where
$$
\langle f, g \rangle = \sum_{x} f(x)g(x)\,\mm(x)\,.
$$
It is well-known that the spectral radius (= norm) of this operator is $\rho$, and that 
$\spec(P)$ is symmetric around the origin.
Now, for $\la \in \C$ with $|\la| > \rho$, the function
$G(x,y|\la)$ is the $(x,y)$-matrix element of the resolvent operator 
$\mathfrak{G}(\la) = (\la\,I - P)^{-1}$ on $\ell^2(T,\mm)$.
Therefore, it extends analytically to  the resolvent set
$\res(P) \subset \C \setminus [-\rho\,,\,\rho]$.

In particular, we note at this point that for $\la \in \res(P)$, the function
$G(x,x|\la)$ has no pole, so that by continuity we see from
Lemma \ref{lem:crucial} and its proof that 
\begin{equation}\label{eq:aboverho}
\big|U(x,x|\la)\big| < |\la|\,, \; G(x,y|\la) \ne 0   \AND 
\big|F(x,y|\la)F(y,x|\la)\big| < 1 
\end{equation}
for $x \sim y$ and $|\la| > \rho\,$, 
and also for $|\la| = \rho$ in the $\rho$-transient case.
Now, 
$$
F(x,y|\la) = G(x,y|\la)/G(y,y|\la) \AND U(x,x|\la) = \la - 1/G(x,x|\la)
$$
extend to meromorphic functions in $\la$ on $\res(P)$. 
We conclude from Lemma \ref{lem:crucial} that for neighbours $x,y \in T$, 
\begin{equation}\label{eq:nonzero}
\begin{gathered}
\text{if } \; G(x,x|\la)G(y,y|\la) \ne 0 \quad \text{then} \quad 
U(x,x|\la) \;\text{ and }\;  F(x,y|\la) \;\text{ are analytic at}\; \la\,,\\
F(x,y| \la) \ne 0 \AND F(x,y|\la)F(y,x|\la) \ne 1
\end{gathered}
\end{equation}
The case when $G(x,x|\la_0)=0$ is quite special.
We know from Lemma \ref{lem:crucial}(c) and \eqref{eq:aboverho} 
that $G(x,x|\la) \ne 0$ for $\la \in \C$ with $|\la| > \rho$.
Replacing $F(x,y|\la)$ by $G(x,y|\la)/G(y,y|\la)$ and $F(y,x|\la)$ by $G(y,x|\la)/G(x,x|\la)$,
we can transform the second identity of Lemma \ref{lem:crucial}(c) into
\begin{equation}\label{eq:G}
\frac{G(x,y|\la)}{p(x,y)} = G(x,x|\la)G(y,y|\la)- G(x,y|\la)G(y,x|\la) = \frac{G(y,x|\la)}{p(y,x)}.
\end{equation}
A priori, this holds for $|\la| > \rho$, but by analytic continuation, it must also hold
in all of $\res(P)$. 
Reordering the terms,
\begin{equation}\label{eq:reorder}
G(x,y|\la)\Bigl(\frac{1}{p(x,y)} + G(y,x|\la)\Bigr) = G(x,x|\la)G(y,y|\la)
\end{equation}
Thus, if $G(x,x|\la_0) = 0$ for some $\la_0 \in \res(P)$,
then for any $y \sim x$, 
$$
p(x,y)G(y,x|\la_0) = p(y,x)G(x,y|\la_0) \in \{ 0,  -1\}\,.
$$
Now, for any $\la \in \res(P)$, we have that $G(\cdot,x|\la) \in \ell^2(T,\mm)$
and $PG(\cdot,x|\la) = \la\,G(\cdot,x|\la) - \de_x\,$. 
In particular, when $G(x,x|\la_0) = 0$ (which can only happen for $|\la| < \rho$)
$$
\sum_y p(x,y) G(y,x|\la_0) = -1\,,
$$
whence $p(x,y)G(y,x|\la_0) = p(y,x)G(x,y|\la_0)$ must have value $-1$ for precisely one neighbour 
of $x$, while the value is $0$ for all
other neighbours. In presence of invariance of the transition probabilities under
a transitive action of a group of automorphisms of trees, this fact can be used
to exclude that the degenerate case $G(x,x|\la_0) = 0$ can occur, see further below.\hfill$\square$
\end{rmksan}


\section{The general Poisson-Martin integral representation on trees}\label{sec:intrep}

In view of last considerations on analytic continuation, we now define
$$
\res^*(P) = \{ \la \in \res(P): G(x,x|\la) \ne 0 \; \text{for all}\; x\}.
$$
\eqref{eq:nonzero} tells us that for any $\la \in \res^*(P)$, we can form 
the $\la$-Martin kernel \eqref{eq:Mkernel}, because the denominator does not
vanish. By our assumptions, $PK(\cdot,w|\lambda)$ is well defined as a function 
of the first variable. That is, even at vertices with infinite degree, 
the involved sum is absolutely convergent, and for $\xi \in \partial T$,
equation
\eqref{eq:eigen} is valid even when $\la \in \res^*(P)$ with $|\la| \le \rho$.
Indeed, to be precise, write 
$$
K(x,\xi|\la) = \frac{G(x,v|\la)}{G(o,v|\la)}\,, \quad \text{where }\;
v \in \pi(o,\xi)\,,\; v^- = x \wedge \xi\,,
$$
so that we can compute in $\ell^2(T,\mm)$
\begin{equation}\label{eq:PKe}
P K(x,\xi|\la) = \frac{1}{G(o,v|\la)} \,P\mathfrak{G}(\la)\uno_v(x) = 
\frac{1}{G(o,v|\la)} \,\la\, \mathfrak{G}(\la)\uno_v(x)\,.
\end{equation}
We now consider distributions. For any $x \in T \setminus \{o\}$, is \emph{predecessor} 
$x^-$ is the neighbour of $x$ which is closer to $o$. We write $T_x = T_{o,x}\,$, 
as well as $T_o = T$. A \emph{complex distribution} on the collection of all boundary arcs
$$
\mathcal{F}_o = \{ \partial T_x : x \in T \}
$$
is a set function $\nu: \mathcal{F}_o \to \C$ such that for every $x$,
\begin{equation}\label{eq:measure}
\nu(\partial T_x) = \sum_{y: y^-=x} \nu(\partial T_y)\,. 
\end{equation}
When $\deg(x)=\infty$, we require that this sum converges absolutely. Note that this
does in general not imply that $\nu$ extends to a $\sigma$-additive distribution on the Borel 
$\sigma$-algebra of $\partial T$. However, $\nu$ clearly extends to the collection of 
all boundary arcs $\partial T_{x,y}\,$, where $x, y \in T$ ($x \ne y$). Indeed, if
$o \notin T_{x,y}$ then $\partial T_{x,y} = \partial T_y \in \mathcal{F}_o\,$. If 
$o \in T_{x,y}$ then $\partial T \setminus \partial T_{x,y} = \partial T_{y,x} = \partial T_x$
and we can define $\nu(\partial T_{x,y}) = \nu(\partial T) - \nu(\partial T_x)$. 
In particular, after a change of the root from $o$ to $o'$, the distribution satisfies  
\eqref{eq:measure} also on $\mathcal{F}_{o'}\,$.

\smallskip

A \emph{locally constant function} on $T$ is a function $f: T \to \C$ such that the set
of edges
$$
E_f = \{ [x,y] \in E(T) : f(x) \ne f(y) \} \quad \text{is finite.}
$$
The union of all geodesic segments which connect $o$ to the endpoints of the edges in $E_f$ 
forms a finite subtree $\tau_f$ of $T$. If $\tau$ is any finite subtree of $T$ containing
$o$, and $x \in \tau$, then let 
$$
S_{\tau}(x) = \{ y \in \tau : y^- = x\}\,,
$$
a finite set, possibly empty (the successors of $x$ in $\tau$). If $\tau = \tau_f$ for $f$ as above,
then $f$ constant on each set $T_x \setminus \bigcup \bigl\{ T_y : y \in S_{\tau}(x) \bigr\}$.
In particular, $f$ extends to a continuous function on $\partial T$. We also call the resulting
function on $\partial T$ locally constant. Thus, a function $\varphi$ on $\partial T$
is locally constant if there is a finite subtree $\tau$ of $T$ which contains $o$ such that
\begin{equation}\label{eq:loccost}
\varphi  \; \text{ has constant value on each set }\;   
\partial T_x \setminus \bigcup \bigl\{ \partial T_y : y \in S_{\tau}(x) \bigr\}\,,\; x \in \tau\,.
\end{equation}
We shall denote that value by $\varphi_x$.

This definition is clearly independent of the choice of the root, and if $\tau'$ is another
finite subtree which fulfils the same requirement for $\varphi$, then so does the union
$\tau \cup \tau'$.  We can now define
\begin{equation}\label{eq:int}
 \int_{\partial T} \varphi\, d\nu = \sum_{x \in \tau} \varphi_x\, 
\biggl( \nu(\partial T_x) - \sum_{y \in S_{\tau}(x)}  \nu(\partial T_y) \biggr).
\end{equation}
\eqref{eq:measure} guarantees that this definition does not depend on the specific
choice of the finite tree $\tau$ for which \eqref{eq:loccost} holds. It also does not
depend on the choice of the root $o$. Consequently,
the integral is linear on the vector space of all locally constant functions on $\partial T$. 

\smallskip

We return to the $\la$-Martin kernel, where $\la \in \res^*(P)$.

Now let $x \in T$ and $\pi(o,x) = [o=x_0\,,x_1\,,\dots, x_k=x]$.
Then $x \wedge \xi \in \{ x_0\,,x_1\,,\dots, x_k\}$ for every $\xi \in \partial T$,
and 
$$
K(x,\xi|\la) = K(x,x_i|\la) \quad\text{when}\quad 
\begin{cases} \xi \in \partial T_{x_i}\setminus \partial T_{x_{i+1}}\,,&i \le k-1,\\
              \xi \in \partial T_{x_k}\,,&i=k\,.
\end{cases}
$$  
Thus, $\varphi = K(x,\cdot|\la)$ is locally constant on $\partial T$, and 
we can use $\pi(o,x)$ as a tree $\tau$ for which \eqref{eq:loccost} holds. 
We subsume.

\begin{pro}\label{pro:transf}
For a distribution $\nu$ as above, the function 
$$
h(x) = \int_{\partial T} K(x,\xi|\la)\, d\nu(\xi)
$$
is a $\la$-harmonic function, and
\begin{equation}\label{eq:int'}
\begin{aligned}
h(x) 
&= \sum_{i=0}^{k-1} K(x,x_i|\la) 
\Bigl(\nu(\partial T_{x_i}) - \nu(\partial T_{x_{i+1}})\Bigr) + K(x,x|\la)\,\nu(\partial T_{x})\\
&= K(x,o|\la)\nu(\partial T) + \sum_{i=1}^k \Bigl(K(x,x_i|\la)-   K(x,x_{i-1}|\la)\Bigr)
 \nu(\partial T_{x_i})\,.
\end{aligned}
\end{equation}
\end{pro}

We call $h$ the \emph{Poisson transform} of $\nu$. The fact that $Ph = \la \cdot h$ follows
from \eqref{eq:int'}, using absolute convergence in \eqref{eq:measure} and \eqref{eq:PKe}.

As outlined in the introduction, for real $\la > \rho$, the following theorem 
goes back to \cite{Ca}  in the locally finite case.
The below proof transfers \cite[Thm. 9.37]{W-Markov} to complex $\la$; 
we present its main part here
with additional care regarding absolute convergence in the non-locally finite case. 

\begin{thm}\label{thm:general} Let $\la \in \res^*(P)$, or $\la = \pm \rho$ in
the $\rho$-transient case.
A function $h: T \to \C$ satisfies $Ph = \la \cdot h$ if and only
if it is of the form
$$
h(x) = \int_{\partial T} K(x,\xi|\la)\,d\nu(\xi)\,,
$$
where $\nu$ is a complex distribution on $\mathcal{F}_o\,.$ The distribution $\nu$ is 
determined by $h$, that is, $\nu = \nu^h$, where  
$$
\nu^h(\partial T)= h(o) \AND 
\nu^h(\partial T_x) = F(o,x|\la)\,\frac{h(x)-F(x,x^-|\la)h(x^-)}{1-F(x^-,x|\la)F(x,x^-|\la)}\,,
\; x \ne o\,.
$$
\end{thm}

\begin{proof} Via analytic continuation \eqref{rmk:an}, the following identities
are a consequence of Lemma \ref{lem:crucial}(c).
\begin{equation}\label{eq:recover}
\begin{aligned}
G(x,x|\la)p(x,y) &= \frac{F(x,y|\la)}{1-F(x,y|\la)F(y,x|\la)}\quad\text{and}\\[.2cm]
\la\, G(x,x|\la) &= 1 + \sum_{y: y \sim x}   
\frac{F(x,y|\la)F(y,x|\la)}{1-F(x,y|\la)F(y,x|\la)}\,,
\end{aligned}
\end{equation}
and when $\deg(x)=\infty$, the last sum converges absolutely.
Indeed, the first identity is part of Lemma \ref{lem:crucial}(c).
The second one then follows from the first one by applying the operator identity
$$
P\, \mathfrak{G}(\la) = \mathfrak{G}(\la)\, P = \la \cdot \mathfrak{G}(\la) - I\,.
$$
We now show first that if $h$ is $\la$-harmonic, then $\nu^h$ as defined in the 
theorem is indeed a complex distribution on $\mathcal{F}_o\,$, and $h$ is its Poisson transform.
We start with the identity
$$
\la\,G(x,x|\la)h(x) = \sum_y G(x,x|\la) p(x,y) h(y)\,,
$$ 
and recall that the sum on the right hand side is assumed to converge absolutely, 
when $\deg(x)=\infty$.
Using \eqref{eq:recover}, we rewrite this as
$$ 
\biggl( 1 + \sum_{y: y \sim x}   \frac{F(x,y|\la)F(y,x|\la)}{1-F(x,y|\la)F(y,x|\la)}\biggr) h(x) =
\sum_{y: y \sim x} \frac{F(x,y|\la)}{1-F(x,y|\la)F(y,x|\la)}\, h(y)\,.
$$
Using that the involved sums converge absolutely, we can regroup the terms and
get 
\begin{equation}\label{eq:hidentity}
h(x) = \sum_{y\,:\,y \sim x} F(x,y|\la)\frac{h(y)-F(y,x|\la)h(x)}{1-F(x,y|\la)F(y,x|\la)}\,.
\end{equation}
Convergence is again absolute when $\deg(x)=\infty\,$. 

For $x = o$, the last identity says that
$\nu^h(\partial T) = \sum_{y \sim o} \nu^h(\partial T_y)$.
Suppose that $x \ne o$. Then by \eqref{eq:hidentity}, 
$$
\begin{aligned}
\sum_{y: y^- = x} \nu^h(\partial T_y) 
&= F(o,x|\la) \sum_{y\,:\,y^- = x} F(x,y|\la) \frac{h(y)-F(y,x|\la)h(x)}{1-F(x,y|\la)F(y,x|\la)} \\[.2cm]
&= F(o,x|\la)
\left(h(x) - F(x,x^-|\la)\frac{h(x^-)-F(x^-,x|\la)h(x)}{1-F(x,x^-|\la)F(x^-,x|\la)}\right)\\[.2cm]
&= F(o,x|\la)\frac{h(x)-F(x,x^-|\la)h(x^-)}{1-F(x,x^-|\la)F(x^-,x|\la)}
=\nu^h(\partial T_x)
\end{aligned}
$$
So $\nu^h$ is indeed a signed distribution on $\mathcal{F}_o\,.$
We verify that $\int_{\partial T} K(x,\xi|\la)\,d\nu^h(\xi) = h(x)$.
For $x = o$ this is true. Let $x \ne o$. With the 
notation of \eqref{eq:int'}, we simplify
$$
\Bigl( K(x,x_i|\la) - K(x,x_{i-1}|\la)\Bigr)\,\nu^h(\partial T_{x_i})
= F(x,x_i|\la)\,h(x_i)-F(x,x_{i-1}|\la)\, h(x_{i-1})\,,
$$
whence we obtain
$$
\begin{aligned}
\int_{\partial T} K(x,\xi|\la)\,d\nu^h(\xi) &= K(x,o|\la)h(o) +
\sum_{i=1}^k \Bigl(F(x,x_i|\la)h(x_j)-F(x,x_{i-1}|\la) h(x_{i-1})\Bigr)\\ 
&= F(x,x|\la)h(x) = h(x)\,,
\end{aligned}
$$
as proposed.

Finally, we need to verify that given $\nu$ and its Poisson transform 
$h$, we have $\nu=\nu^h$. We omit this part of the proof, which is precisely as
in \cite[Thm. 9.37]{W-Markov}, with the only generalisation that $Ph = \la\cdot h$,
not necessarliy having $\la =1$.
\end{proof}

\begin{rmksig}\label{rmk:sigma}
If the distribution $\nu$ on $\mathcal{F}_o$ is non-negative, then it extends uniquely
to a $\sigma$-additive measure on the Borel $\sigma$-algebra on $\partial T$.
This is a consequence of measure theoretic basics, such as the extension theorems
of Caratheodory or Kolmogorov. In the context of boundaries of trees, this fact seems less
understood, so we give an outline. Let $\mathcal{A}$ be 
the product $\sigma$-algebra on $\Ss^{\N}$. It is generated by all cylinder sets
$$
C(\as_1\,,\dots,\as_k) = \{ (\sr_n) \in \Ss^{\N} : \sr_i = \as_i \,,\; i=1,\dots, k \}\,, 
$$
where $k \ge 1$ and $\as_1\,,\dots,\as_k \in \Ss$. We also allow $k=0$ with the
corresponding cylinder set being all of $\Ss^{\N}$.

Now suppose that for each $k \in \N$, we have a probability distribution $p_k$ on 
$\Ss^k$, such that for all $k$ and all $(\sr_1\,,\dots\,,\sr_k) \in \Ss^k$,
\begin{equation}\label{eq:pk}
p_k(\sr_1\,,\dots\,,\sr_k) = \sum_{\sr_{k+1} \in \Ss} p_{k+1}(\sr_1\,,\dots\,,\sr_k\,,\sr_{k+1})\,.
 \end{equation}
Then there is a unique probability measure $\Prob$ on $\Ss^{\N}$, such that for
every cylinder set, 
$$
\Prob\bigl(C(\as_1\,,\dots,\as_k)\bigr) = p_k(\as_1\,,\dots\,,\as_k)\,.
$$
A nice exposition, given in the context of Martin boundary theory, is  
due to {\sc Dynkin}~\cite{Dy}. There are of course various other sources.

How is the relation with trees$\,$? Let $\Ss^* = \bigcup_{k=0}^{\infty} \Ss^k$
be the collection of all words over $\Ss$. For $k=0$, we intend that $\Ss^0$
consists of the \emph{empty word $o$,} and $\Ss^k$ is the collection of all words
of length $k$. We put a tree structure on $T = \Ss^*$, where the root is the empty word
$o$, and neighbourhood is defined in terms of predecessors: for $k \ge 1$,
the predecessor of $x= (\sr_1\,,\dots\,,\sr_k)$ is $x^- = (\sr_1\,,\dots\,,\sr_{k-1})$.
Then $\partial T = \Ss^{\N}$ and $\partial T_x = C(\sr_1\,,\dots,\sr_k)$. 
With this identification, a sequence of probability distributions $p_k$ on 
$\Ss^k$ which satisfy $\eqref{eq:pk}$ corresponds bijectively to a measure $\nu$ 
on $\mathcal{F}_o$ as in \eqref{eq:measure} by
$$
\nu(\partial T_x) = p_k(\sr_1\,,\dots\,,\sr_k)\,,\quad \text{where } \; 
x=  (\sr_1\,,\dots\,,\sr_k) 
$$
If we start with a countable tree $T$ where $\deg(x)=\infty$ for every vertex,
then we can choose a root $o$ and label the successors of any vertex by
the natural numbers. In this way, we identify $T$ with $\N^*$.

A tree which has vertices with finite degree can be embedded into $\N^*$ in
an obvious way, and then $p_k$ will assign value $0$ to sequences $(\sr_1\,,\dots\,,\sr_k)$
which do not correspond to vertices of the tree. 
This explains that for any countable tree, any non-negative 
distribution on $\mathcal{F}_o$ extends to a Borel measure on the boundary of the tree.

\medskip

Finally, we may ask when a complex distribution $\nu$ on $\mathcal{F}_o$ which satisfies \eqref{eq:measure}
extends to a $\sigma$-additive Borel measure on $\partial T$. In the locally finite case,
a necessary and sufficient criterion was given by {\sc Cohen, Colonna and Singman}~\cite{CoCoSi},
which generalises to the non-locally finite case as follows: 
finite total variation of $\nu$ is equivalent with existence of some $M < \infty$ 
such that for any sequence
of pairwise disjoint boundary arcs $\partial T_{x_n}$ (and not only those where all $x_n$ have
the same predecessor), one has 
\begin{equation}\label{eq:M}
\sum_n |\nu(\partial T_{x_n})| \le M \,.
\end{equation}
(In the locally finite case, it is sufficient that any such series converges, and the
upper bound follows.)  
We omit the details.~\hfill$\square$
\end{rmksig}


\section{Invariance under general transitive group actions}\label{sec:invariance}

We now study nearest neighbour random walks on countable trees which are invariant
under a general subgroup $\Gamma$ of the automorphism group of $T$ which acts transitively 
on the vertex set. Invariance means that $p(\gamma x,\gamma y) = p(x,y)$ for all 
$x,y \in T$ and $\gamma \in \Gamma$. 
Let $I = \Gamma \backslash E(T)$ be the set of orbits of $\Gamma$ on the set of oriented
edges of $T$. If $j \in I$ is the orbit \emph{(type)} of $(x,y) \in E(T)$ then we write 
$p_j=p(x,y)$ and
$-j$ for the orbit of $(y,x)$. This is independent of the representative $(x,y)$, 
and defines an involution:
$-(-j) = j$. We have $-j = j$ if and only if there is $\gamma \in \Gamma$ which inverts
$(x,y)$. 
For each $j \in I$ and fixed $x \in T$, we set $d_j = |\{ y \sim x : (x,y) \in I \}|$.
This is independent of $x$ by transitivity of $\Gamma$.
When $-j \ne j$ we may well have  $d_{-j} \ne d_j$. (For example, we may have
$I = \{ \pm 1 \}$ with $d_{-1} = 1$ and $d_{1} = q \in \N$.)
The degree of any vertex $x$ is 
$$
\deg(x) = \sum_{j \in I} d_j\,, \AND  \sum_{j \in I} d_j\,p_j =1 \,.
$$
In particular, each $d_j$ is finite. 
Conversely, if we start with a finite or countable set $I$ with an involution 
$j \mapsto -j$ and a collection $(d_j)_{j \in I}$ of natural numbers, then for 
the regular tree $T$ with degree $\sum_j d_j \le \infty\,$, there is    
a group $\Gamma \le \Aut(T)$ which acts transitively 
such that $I$ is its set
of orbits and the associated cardinalities are $d_j\,$. We postpone
the explanation of this 
and a few other group-theoretic facts to the end of the present section. 
As a specific example, when $d_j=1$ 
for all $j$ then we can take the discrete group
\begin{equation}\label{eq:discgp}
\Gamma = \langle a_j\,, j \in I \mid 
a_j^{-1} = a_{-j} \;\text{for all}\; j \in I \rangle\,.
\end{equation}
When $j \ne -j$ then we can take just one out of $a_j$ and
$a_{-j}$ as a free generator; when $j = -j$ then $a_j$
is a generator whose square is the group identity. In this example, $\Gamma$ acts transitiviely
with trivial stabilisers, and the fact that this provides all possible groups which act 
in this way on a countable tree is a well-known basic part of Bass-Serre Theory, 
see {\sc Serre}~\cite{Se}. 

Here, we shall always assume that the vertex degree is $\ge 3$, so that our random walk
has to be $\rho$-transient by a result of {\sc Guivarc'h}~\cite{Gu}. 
The Green function $G(x,x|\la) = G(\la)$ is again 
independent of $x$, and  if $(x,y)$ is an oriented edge of type $j$, then $G(x,y|\la) = G_j(\la)$ 
depends only on $j$. 
If $G(\la_0) = 0$ for $\la_0 \in \res(P)$ then there must be a unique
$j_0 \in I$ such that 
$$
-j_0=j_0\,,\quad d_{j_0}=1\,,\quad p_{j_0}\,G_{j_0}(\la_0) = -1\,, \AND 
p_j G_{-j}(\la_0)=0 \quad \text{for all}\; j \ne j_0\,.
$$
This follows from the last lines of  Remarks
\ref{rmk:an}. Namely, if we
fix $x \in T$ then there is precisely one $y \sim x$ such that 
$p(x,y)G(y,x|\la_0) = p(y,x)G(x,y|\la_0) = -1$, 
while $p(x,v)G(v,x|\la_0) = p(v,x)G(x,v|\la_0) = 0$ for all other $v \sim x$.

Thus, if $d_j \ge 2$ for all $j$ with $-j = j$, then $G(\la) \ne 0$ for all 
$\la \in \res(P)$. 

\begin{cor}\label{cor:free}
For any nearest neighbour random walk on a finitely or countably generated 
free group, $G(\la) \ne 0$ for all $\la \in \res(P)$, so that $\res^*(P)=\res(P)$. 
\end{cor}

By reversibility, we have
\begin{equation}\label{eq:reversibility}
p_j\, G_{-j}(\la) = p_{-j}\, G_j(\la)\,,
\end{equation}
and \eqref{eq:reorder} becomes 
\begin{equation}\label{eq:GG}
p_{-j}\, G_j(\la)^2 + G_j(\la) - p_j\, G(\la)^2 = 0\,.  
\end{equation}
When $\la > \rho$ is real, the right one among the two solutions of this equation is
\begin{equation}\label{eq:Gj}
G_j(\la) = \frac{1}{2p_{-j}}\Bigl(\sqrt{1 + 4p_jp_{-j} \,G(\la)^2} - 1\Bigr)\,,
\end{equation} 
because the functions $G(\la)$ and $G_j(\la)$ are decreasing in those $\la$.
In other regions of the plane, there may be the minus sign in front of the root.
We now show that the only possible $\la_0 \in \res(P)$ for which 
$G(\la_0) = 0$ is $\la_0=0$. 

\begin{thm}\label{thm:nonzero}
If $0 \ne \la \in \res(P)$ then $G(\la) \ne 0$, so that $\res(P) \setminus \res^*(P) \subset \{0\}$.
\end{thm}

\begin{proof}
We modify and extend an argument of \cite[p. 17]{FiSt}.  
Suppose that $G(\la_0) = 0$ for
some $\la_0 \in \res(P)$. By analyticity, there is an open ball $B$ centred
at $\la _0$ with closure
$B^- \subset \res(P)$ such that $G(\la) \ne 0$ for $\la \in B^- \setminus \{\la_0\}$.
By continuity, we may also choose $B$ such that $2|G(\la)| < 1$ for all $\la \in B^-$.
 
We know from \eqref{rmk:an} that \eqref{eq:nonzero} holds for those $\la$, and 
as we have observed above, there is precisely one $j_0 \in I$ such that 
$p_{-j_0}\, G_{j_0}(\la_0) = -1$, andby \eqref{eq:reversibility} we must have $-j_0=j_0$ and $d_{j_0} = 1$. 

For all $j \in I_0 = I \setminus\{j_0\}$,
we have $G_j(\la_0)=0$, so that the  
right solution of \eqref{eq:GG} at $\la_0$ is the one given by \eqref{eq:Gj}.
By continuity, this must also hold in a neighbourhood of $\la_0\,$. On the other 
hand, by the choice of $B$, the right hand side of \eqref{eq:GG} is analytic in 
$B$, so that by analytic continuation, \eqref{eq:GG} must hold for $G_j(\la)$ in all of
$B$. Analogously, 
$$
G_{j_0}(\la) =  \frac{1}{2p_{-j}}\Bigl(-\sqrt{1 + 4p_jp_{-j} \,G(\la)^2} - 1\Bigr) 
\quad \text{for }\;\la \in  B\,.
$$
We now note that for complex $t$ with $|t| < 1$, one has $\big| \sqrt{1+t} -1 \big| < |t|$.
Thus, for $\la  \in B$ and $j \in I_0$, \eqref{eq:Gj} gives
$|G_j(\la)| \le 2p_j\,|G(\la)|^2$. When in addition $\la \ne \la_0$, 
setting $F_j(\la) = G_j(\la)/G(\la)$, we can transform the identity of Lemma 
\ref{lem:crucial}(b) into
$$
\la = \frac{G(\la)}{G_{j_0}(\la)} + \sum_{j \in I_0} d_{-j}p_{-j}F_j(\la)\,.
$$  
By the above, for $\la \in B$,
$$
\sum_{j \in I_0} d_{-j}p_{-j}|F_j(\la)| \le 2|G(\la)| \sum_{j \in I_0} d_{-j}p_{-j}p_j 
< 2|G(\la)|\,,
$$
since $d_{-j}p_{-j} \le 1$. Thus, when $\la \to \la_0$, the sum tends to $0$,
and we also know that $G(\la)/G_{j_0}(\la) \to 0$.
Therefore $\la_0=0$, which we have excluded.
\end{proof}

\begin{cor}\label{cor:gr}
In case of group-invariance under a transitive automorphism group of $T$,
the integral representation of Theorem \ref{thm:general} holds for
all eigenvalues $\la \in \res(P)$ with the only possible exception of $\la=0$. 
\end{cor}

We remark that it may well happen that $0$ 
is part of the resolvent set of $P$ for certain choices of the probabilities 
$p_j\,$, see \cite{FiSt}.

\begin{rmkgp}\label{rmks:gp}
We now add some details on group actions. For any tree (or graph), the automorphism 
group $\Aut(T)$ of all neighbourhood preserving bijections of the vertex set is
equipped with the topology of point-wise convergence. If $T$ is locally finite,
then the vertex stabilisers $\{ \gamma \in \Aut(T) : \gamma x = x \}$ for $x \in T$
are open-compact, so that $\Aut(T)$ is a locally compact, totally disconnected 
group. When $T$ is not locally finite, this is more delicate. However, given
the transition operator, resp. matrix $P= \bigl(p(x,y)\bigr)_{x,y \in T}\,$,
we are interested in the group 
$$
\Aut(T,P) = \{ \gamma \in \Aut(T) : p(\gamma x,\gamma y) =  p(x,y) \;\text{ for all }\; x,y \}\,.
$$
When $\Aut(T,P)$ acts transitively (even when $T$ is not locally finite), it is closed in the topology
of point-wise convergence and acts with open-compact vertex stabilisers whose orbits are all
finite, see {\sc Kaimanovich and Woess}~\cite[Prop. 2.7]{KaWo}. 
Hence, if $\Gamma$ is a group as assumed at the beginning of this section, 
we can pass to its closure in $\Aut(T,P)$, or rather assume right away without
loss of generality that $\Gamma = \Aut(T,P)$.

We now address the converse question: we start with a finite or countable set $I$ with an involution 
$j \mapsto -j$ and a collection $(d_j)_{j \in I}$ of natural numbers and look for a
group $\Gamma \le \Aut(T)$ which acts transitively and realises $I$ as the set
of orbits and the associated cardinalities $d_j\,$. 

To begin, we fix $j$ together with $d_j$ and $d_{-j}\,$. 

Let us first consider the case when under the given involution, $-j \ne j$. 
Then we consider the regular tree $T_j = T(d_j,d_{-j})$ with degree $d_j+d_{-j}\,$.
We orient its edges such that each vertex has $d_{-j}$ ingoing and $d_j$ outgoing
edges. All edges are given the ``colour'' $j$.
Then we consider the group $\Gamma_j$ of all automorphisms of $T_j$
which preserve the orientation. This is a closed subgroup of $\Aut(T_j)$ which
acts transitively, and the stabiliser of any vertex $x$ has two sub-orbits
on the neighbours of $x$, namely the forward and the backward neighbours.

Next, consider the special case when $j = - j$ and thus $d_j = d_{-j}$.
Then we take $T_j$ to be the regular tree with degree $d_j$ and $\Gamma_j = \Aut(T_j)$,
its full automorphsim group. (Note that when $d_j=d_{-j}=1$, the tree $T_j$ consists of
two vertices connected by one edge.) The stabiliser of any vertex acts transitively
on its neighbours, which is what we need.

Now, we consider the reduced index set $I'$, where for each $j \in I$ with $-j \ne j$,
we keep only one of $j$ and $-j$ and eliminate the other one. At this point, we
consider the \emph{free product graph} 
$$
T = \freeprod_{j \in I'} T_j\,.
$$
For this purpose, choose a root $o_j$ in each tree $T_j$ and set $T_j' = T_j \setminus \{o_j\}$.
The vertex set of $T$ then consists of all \emph{alternating words} over
the set $\bigcup_{j\in I'} T_j'\,$, that is all finite sequences 
\begin{equation}\label{eq:free}
x_1x_2\cdots x_n\,, \quad\text{where}\; n \ge 0\; \text{and}\; x_k \in T_{j(k)}'\;
\text{for some}\;j(k) \in I'\; \text{with}\;j(k+1) \ne j(k).
\end{equation}
For $n=0$, this is the empty word $o$. The edges of $T$ are as follows for all $n \ge 1$.

\begin{enumerate}
\item[(a)] A vertex $x_1\cdots x_{n-1}x_n$ as above is connected by an edge with colour
$j(n)$ to  $x_1\cdots x_{n-1}y_n$ if and only if there is an edge from $x_n$ to $y_n$ in 
$T_{j(n)}$.

\item[(b)] A vertex $x_1\cdots x_{n-1}$ is connected by an edge with colour
$j(n)$ to $x_1 \cdots x_{n-1}x_n$ if and only if there is an edge from $o_{j(n)}$
to $x_n$ in $T_{j(n)}$.
\end{enumerate}
In both cases, the resulting edges inherit their eventual orientation from $T_j\,$.

Thus, $T$ is a tree which contains countably many copies of each $T_j\,$: any
element $x_1\cdots x_{n-1}$ with $j(n-1) \ne j$ together with the collection of
all $x_1 \cdots x_{n-1}x_n$ with $x_n \in T_j'$ spans such a copy, in which the
root of $T_j$ corresponds to $x_1\cdots x_{n-1}\,$.

Each $\Gamma_j$ embeds into $\Aut(T)$ as follows. 
Let $\gamma \in \Gamma_j$ and $x_1\cdots x_n \in T$ as in \eqref{eq:free}.

\begin{itemize}
\item If $j(1) \ne j$ and $\gamma o_j = o_j$ then $\gamma(x_1\cdots x_n) = x_1\cdots x_n\,$,
in particular $\gamma o = o$ in $T$.

\item If $j(1) \ne j$ and $\gamma o_j \ne o_j$ then $\gamma(x_1\cdots x_n) = 
(\gamma o_j) x_1\cdots x_n\,$.

\item If $j(1) = j$ and $\gamma x_1 = o_j$ then $\gamma(x_1\cdots x_n) = x_2\cdots x_n\,$.

\item If $j(1) = j$ and $\gamma x_1 \ne o_j$ then 
$\gamma(x_1\cdots x_n) = (\gamma x_1)x_2\cdots x_n\,$.
\end{itemize}

The subgroup of $\Aut(T)$ generated by $\bigcup_{j \in I'} \Gamma_j$  preserves the colours
and the eventual orientations of the edges. It acts transitively and has
all the required properties.\hfill$\square$
\end{rmkgp}

\section{The integral representation of $\la$-polyharmonic functions}\label{sec:poly}

We return to the general setting of \S \ref{sec:basic}. A \emph{$\la$-polyharmonic function
of order $n \ge 1$} is a function $f: T \to \C$ such that 
$$
(\la\,I-P)^n\, f \equiv 0\,.
$$
If $n=1$, this means that $f$ is $\la$-harmonic. If $n=2$ this means that $\la\cdot f-Pf$
is $\la$-harmonic, and so on. 

In the setting of the classical Laplacian $\Delta$ on
a Euclidean domain, polyharmonic functions are functions for which $\Delta^n h \equiv 0$.
Their study goes back to work in the $19^{\text{th}}$ century, see e.g. 
{\sc Almansi}~\cite{Al}. A basic reference is the monograph by {\sc Aronszajn, 
Creese and Lipkin}~\cite{ACL}, and there is ongoing study.

The discrete analogue of polyharmonic functions on trees ($\la = 1$)
was studied in a long paper by {\sc Cohen et al.}~\cite{CCGS}. For the case of simple random
walk on a homogeneous tree, they provide a boundary integral representation for
polyharmonic functions. Here, we shall explain how this can by achieved more directly 
for arbitrary $\la$-polyharmonic functions in the general setting of a nearest neighbour 
random walk on a countable tree $T$ (locally finite or not), whenever  
$\la \in \res^*(P)$. 
We know from Lemma \ref{lem:crucial} and the observation of \eqref{rmk:an}  
that all $\la \in \C \setminus \{\pm \rho\}$ 
with $|\la| \ge \rho$ belong to $\res^*(P)$,  
and we also know from Theorem \ref{thm:nonzero}  
that in the group-invariant case $\res^*(P)$ differs from the resolvent set 
$\res(P)$ at most by $\{ 0\}$.

The key to our integral representation is the following proposition, 
where for $x \in T$ and $\xi \in \partial T$,
we denote by $K^{(r)}(x,\xi|\la)$ the $r^{\textrm{th}}$ derivative of  the function 
$\la \mapsto K(x,\xi|\la)$. Differentiability of that function follows from
\eqref{eq:Mkernel}, i.e., the fact that it is the quotient of the Green functions at vertices of 
$T$ -- a property which 
is specific to the nearest neighbour case.

\begin{pro}\label{pro:derive} For $\la \in \res^*(P)$,
$|\la| > \rho$, 
$$
(\la\, I - P) K^{(r)}(\cdot,\xi|\la) = r\,  K^{(r-1)}(\cdot,\xi|\la)\,.
$$
Therefore, setting
$$
h(x|\la) = \int_{\partial T} K(\cdot,\xi|\la) \,d\nu(\xi)\,, 
$$
its $r^{\text{\rm th}}$ derivative 
$h^{(r)}(x|\la)$ with respect to $\la$ satisfies
$$
(\la\, I - P) h^{(r)}(\cdot\,|\la) = r\,  h^{(r-1)}(\cdot\,|\la)\,.
$$
\end{pro}

\begin{proof}
Given $x$ and $\xi$, let $c =  x \wedge \xi$. We can write 
$K(x,\xi|\la) = G(x,c|\la)/G(o,c|\la)$, which is legitimate since
our $\la$ is such that the denominator does not vanish. 
When $T$ is locally finite, it is obvious that in the identity
$$
P\, K(x,\xi|\la) = \la\, K(x,\xi|\la)\,,
$$
we may exchange the derivatives with respect to $\la$ with the application of 
$P$, even when $T$ is not locally finite. The fact that this is also true in the
non-locally finite case follows from Lemma \ref{lem:resolvent} 
below, together with \eqref{eq:PKe}. 

Next, formula \eqref{eq:int'} of Proposition \ref{pro:transf}  
shows that $\; \int_{\partial T} K(x,\xi|\la) \,d\nu(\xi)\;$ 
is in fact a finite linear combination of functions $K(x,x_i|\la)$, so that the second 
statement is immediate.
\end{proof}

Let us now clarify why derivation and summation in \eqref{eq:PKe} can be exchanged even
in the non-locally finite case. We recall (without proof) some basic facts which are part of the 
functional calculus for the resolvent
of a bounded self-adjoint operator on a Hilbert space within our context; see e.g.
{\sc Kato}~\cite{Ka}, {\sc Yosida}~\cite[Ch. VIII]{Yo},  or various other textbooks.

\begin{lem}\label{lem:resolvent} For $\la \in \res(P)$,
we have the following for the resolvent operator $\mathfrak{G}(\la)= (\la\,I - P)^{-1}$
on $\ell^2(T,\mm)$.

The operator-valued mapping $\la \mapsto \mathfrak{G}(\la)$ is analytic, and its
$r^{\textrm{th}}$ derivative $\mathfrak{G}^{(r)}(\la)$ with respect to $\la$ satisfies
$$
\begin{aligned}
\mathfrak{G}^{(r)}(\la) &= (-1)^r\, r! \, \bigl(\mathfrak{G}(\la)\bigr)^{r+1} 
\; \text{(operator power of order $r+1$),\; and}\\
P\, \mathfrak{G}^{(r)}(\la) &= \frac{d^r}{d\la^r} \bigl(P\, \mathfrak{G}(\la)\bigr)
= r \cdot \mathfrak{G}^{(r-1)}(\la) + \la\cdot \mathfrak{G}^{(r)}(\la)\,,\quad  r \ge 1.
\end{aligned}
$$
\end{lem}

\noindent
For $r=1$, the first formula may be seen as a consequence of the \emph{resolvent equation}
$$
\mathfrak{G}(\la_1) - \mathfrak{G}(\la_2) 
= (\la_2-\la_1)\cdot \mathfrak{G}(\la_1)\mathfrak{G}(\la_2)\,.
$$
Subsequently, one can proceed by induction on $r$. Reading things element-wise, we have for example 
$$
\begin{aligned}
G'(x,y|\la) &= - \sum_v G(x,v|\la)G(v,y|\la) \AND\\ G''(x,y|\la) &= 
2 \sum_{v, w} G(x,v|\la)G(v,w|\la)G(w,y|\la)
\end{aligned}
$$
as well as
$$
\sum_v p(x,v)\, G'(v,y|\la) = G(x,y|\la) + \la\, G'(x,y|\la),
$$
etc. At this point, wee see that Proposition \ref{pro:derive} is valid in general.

If $h$ is a $\la$-harmonic function then by Theorem \ref{thm:general} there is a unique
distribution $\nu$ on $\partial T$ such that
$$
h(x) = h(x|\la) = \int_{\partial T_x} K(x,\xi|\la)\,d\nu(\xi)\,.
$$
and we can consider the functions 
$$
h^{(r)}(x) = h^{(r)}(x|\la) = \int_{\partial T_x} K^{(r)}(x,\xi|\la)\,d\nu(\xi)\,.
$$
From Proposition \ref{pro:derive} we infer that
$$
(\la \, I - P)^r h^{(r)} = r! \cdot h\,.
$$
So, since $h$ is $\la$-harmonic, the derived function 
$h^{(r)}$ is $\la$-polyharmonic of order $r+1$.
This leads us to the main theorem of this section.

\begin{thm}\label{thm:poly}
For $\la \in \res^*(P)$, every $\la$-polyharmonic function $f$ of order $n$
has an integral representation
$$
f(x) = \sum_{r=0}^{n-1} \int_{\partial T} K^{(r)}(x,\xi|\la)\,d\nu_r(\xi)\,,
$$
where the collection of distributions $(\nu_0\,,\dots, \nu_{n-1})$ in the sense
of \eqref{eq:measure} is uniquely determined by $f$.  
Conversely, every function which has an integral representation as above
is $\la$-polyharmonic of order $n$.
\end{thm}

\begin{proof} We use induction on $n$. The statement is true for $n=1$ by Theorem
 \ref{thm:general}. Now suppose it is true for $n-1$. Let $f$ be polyharmonic of
order $n$. Then we have the integral representation of the $\la$-harmonic
function
$$
h(x) = \frac{1}{(n-1)!}(\la\,I - P)^{n-1}f 
= \int_{\partial T}  K(x,\xi|\la)\,d\nu_{n-1}(\xi)
$$
for the uniquely determined distribution $\nu_{n-1}$ on $\partial T$. Deriving $n-1$ times
with respect to $\la$, Proposition \ref{pro:derive} yields
$$
(\la\, I - P)^{n-1} h^{(n-1)} = (n-1)!\, h = (\la\,I - P)^{n-1}f\,.
$$
Thus, the function $f - h^{(n-1)}$ is $\la$-polyharmonic of order $n-1$. By the induction hypothesis,
it has a unique integral representation
$$
f(x) - h^{(n-1)}(x) = \sum_{r=0}^{n-2} \int_{\partial T} K^{(r)}(x,\xi|\la)\,d\nu_r(\xi)\,.
$$
Since
$$
h^{(n-1)}(x) = \int_{\partial T}  K^{(n-1)}(x,\xi|\la)\,d\nu_{n-1}(\xi)\,,
$$
the result follows.
\end{proof}

\textbf{The isotropic case}

We now consider the specific case when $T = \T_q$ is the homogeneous tree with degree
$q+1$, and $P$ is simple random walk, i.e., $p(x,y)=1/(q+1)$ when $x \sim y$.
In this situation, \cite[Thm. 4.1]{CCGS} gives an integral representation
of polyharmonic functions for $\la=1$. By making use of Theorem \ref{thm:poly} above, 
we give a simpler proof of that result, and we also extend it
to general eigenvalues $\la$. First of all, it is well known
at least since \cite{Ca} (and has been computed again and again by many authors) 
that $\spec(P) = [-\rho\,,\,\rho]$, where $\rho = 2\sqrt{q}/(q+1)$.
Thus, by Theorem \ref{thm:nonzero}, $\res^*(P) = \res(P) = 
\C \setminus [-\rho\,,\,\rho]$.
For $\la \in \res(P)$ and neighbours $x,y \in T$, the function $F(\la) = F(x,y|\la)$ 
has also been re-computed many times, 
usually in the variable $z=1/\la$, 
on the basis of the quadratic equation which results from Lemma \ref{lem:crucial}(b). 
The expression of the $\la$-Martin kernel becomes
$$
K(x,\xi|\la) =  F(\la)^{\hor(x,\xi)}\,,\quad\text{where}\quad
F(\la) = F_-(\la) = \frac{(q+1)\la}{2q}\Bigl( 1 - \sqrt{1 - \rho^2/\la^2}\Bigr)\,,
$$
and $\hor(x,\xi) = d(x,x\wedge \xi) - d(o,x\wedge \xi)$ is the \emph{horocycle index} 
of $x$ with respect to $\xi$. We compute 
$$
K'(x,\xi|\la) = K(x,\xi|\la)\, \hor(x,\xi) \, f(\la)\,, \quad
\text{where} \quad f(\la) = -\frac{1}{\la \sqrt{1- \rho^2/\la^2}}\,. 
$$
From here, we get recursively
\begin{equation}\label{eq:rec}
\begin{aligned}
&K^{(r)}(x,\xi|\la) = K(x,\xi|\la)\, \sum_{k=1}^r \hor(x,\xi)^k \, f_{k,r}(\la)\,, \quad
\text{where}\\
&f_{1,r}(\la) = f^{(r-1)}(\la)\quad [\text{derivative of order}\; r-1]\,,
\quad f_{r,r}(\la) = f(\la)^r 
\,,\quad\text{and}\\
&f_{k,r}(\la) = f_{k,r-1}'(\la) + f(\la) f_{k-1,r-1}(\la)
\quad \text{for }\; r \ge 2\,,\; k=2,\dots,r-1\,.
\end{aligned}
\end{equation}
Now suppose that $f$ is a $\la$-polyharmonic function of order $n$ for our simple random walk
on the homogeneous tree. Let
$$
f(x) = \int_{\partial T} K(x,\xi|\la)\, d\nu_0(\xi) +  
\sum_{r=1}^{n-1} \int_{\partial T} K^{(r)}(x,\xi|\la)\,d\nu_r(\xi)
$$
be its unique integral representation according to Theorem \ref{thm:poly}.
By \eqref{eq:rec}, we can rewrite the last sum in terms of new 
distributions $\bar\nu_k$ as
$$
\sum_{k=1}^{n-1} \int_{\partial T} K(x,\xi|\la)\, \hor(x,\xi)^k\, d\bar\nu_k\,, 
\quad\text{where}\quad
\bar \nu_k = \sum_{r=k}^{n-1} f_{k,r}(\la)\, \nu_r\,.
$$
Since $f(\la) \ne 0$, the upper triangular $(n-1) \times (n-1)$-matrix 
$A_{n-1}(\la) = \bigl( f_{k,r}(\la) \bigr)_{1 \le k \le r \le n-1}$ is invertible. 
Thus, we can restate Theorem \ref{thm:poly} in this case as follows.

\begin{cor}\label{cor:poly}
For simple random walk on $T = \T_q$ and $\la \in \C \setminus [-\rho\,,\,\rho]$,
let $f$ be a $\la$-polyharmonic function of order $n$. Then $f$ has an 
integral representation
$$
f(x) = \sum_{k=0}^{n-1} \int_{\partial T} K(x,\xi|\la)\,\hor(x,\xi)^k\, d\bar\nu_k(\xi)\,,
$$
where the collection of distributions $(\bar \nu_0\,,\dots, \bar\nu_{n-1})$ in the sense
of \eqref{eq:measure} is uniquely determined by $f$.  
\end{cor}

(Note that $\bar \nu_0 = \nu_0\,$.) Uniqueness follows from invertibility 
of $A_{n-1}(\la)$. In the special case $\la =1$, this corollary is \cite[Thm. 4.1]{CCGS}.
Indeed, note that for simple random walk on $\T_q\,$, the $k_{\omega}(v)$ of \cite{CCGS} is
$-\hor(v,\omega)$ in our notation, where $v \in T$ and $\omega \in \partial T$.
Also, $F(1) = 1/q$, so that Corollary \ref{cor:poly} yields 
the representation of  \cite[Thm. 4.1]{CCGS}. 

Next, in 
forthcoming work we shall study the interplay of the boundary 
integral representation with the limiting behaviour of polyharmonic functions in
our general setting.

\section{Appendix: remarks on ``forward only'' transition operators}\label{sec:forward}

So far we have assumed $p(x,y)>0$ whenever $x\sim y$.
On a countable tree $T$ with root vertex $o$ as before, we now consider a different, simpler
type of random walks. To 
keep notation different, we denote the stochastic transition 
matrix by
$Q = \bigl( q(x,y) \bigr)_{x,y \in T}\,$, assuming that
$$
q(x,y) > 0 \iff x = y^-\,.
$$
The associated random walk $(Y_n)_{n \ge 0}$ starting at $o$ is such that $d(Y_n,o)=n$.
The general $n$-step-transition probabilities are 
\begin{equation}\label{eq:qn}
\begin{gathered}
q^{(n)}(x,y) = q(x_0\,,x_1)\,q(x_1\,,x_2)\cdots q(x_{n-1}\,,x_n)\,, \\
\text{if }\; y \in T_x \AND \pi(x,y) = [x=x_0\,,x_1\,,\dots, x_n=y]\,,
\end{gathered}
\end{equation}
while $q^{(n)}(x,y)=0$ in all other cases. Of course, we define $q^{(0)}(x,y) = \delta_x(y)$. 
It is clear that $(Y_n)$ converges almost surely to
a $\partial T$-valued random variable $Y_{\infty}\,$, whose distribution is the Borel measure
on $\partial T$ given by 
\begin{equation}\label{eq:bnu}
\boldsymbol{\nu}(\partial T_x) = q^{(n)}(o,x)\,,\quad \text{where }\; n = |x|\,.
\end{equation}
(Recall that $|x| = d(x,o)$.) In fact, this is precisely the situation defined in \eqref{rmk:sigma},
and we can recover the transition probabilities from $\boldsymbol{\nu}$ by
$$
q(x,y) = \boldsymbol{\nu}(\partial T_y)/\boldsymbol{\nu}(\partial T_x)\,,\quad \text{when }\; x = y^-\,.
$$
For $\la \in \C$, a $\la$-harmonic function $h$ is defined as before: $Qh = \la \cdot h$, or
equivalently, 
$$
\sum_{y\,:\, y^- = x} \boldsymbol{\nu}(\partial T_y) \, h(y) = \boldsymbol{\nu}(\partial T_x) \, h(x)
\quad \text{for every }\; x \in T\,.
$$
For $\la=1$, such functions are often called \emph{martingales,} which is completely justified.
Indeed, keeping in mind \eqref{rmk:sigma}, we can view the random variables $Y_n$ to be defined
on the probability space $\bigr( \partial T, \boldsymbol{\nu} \bigr)$ by $Y_n(\xi) = y_n\,,$
where $y_n$ is the $n^{\textrm{th}}$ vertex on the ray $\pi(o,\xi)$. The filtration 
$(\mathcal{F}_n)_{n \in \N}$ of the Borel 
$\sigma$-algebra of $\partial T$ induced by the stochastic process $(Y_n)$ is the one where
$\mathcal{F}_n$ is the sub-$\sigma$-algebra generated by the collection of atoms
$\{ \partial T_x : |x|=n \}$.  
Then the martingales with respect to this filtration are precisely the sequences $\bigl(h(Y_n)\bigr)$,
where $Qh = h$; compare with \cite{Dy} or with \cite[\S 7.C.I]{W-Markov}.
Analogously, $\la$-harmonic functions correspond to $\la$-martingales,
where 
$$
\Ex_{\boldsymbol{\nu}}[ h(Y_{n+1}) \mid \mathcal{F}_n ] = \la\, h(Y_n)\,.
$$
For $\la > 0$, the Martin boundary theory for non-negative $\la$-harmonic functions with respect to $Q$ 
works in the same way as for irreducible Markov chains, since every $x \in T$ can be reached from $o$
with positive probability. Indeed, in analogy with Theorem \ref{thm:general}, it
works for all $\la \ne 0$.

In our situation, the random walk starting at $x \in T$ can visit a vertex $y$ at most once, namely
at time $n=d(x,y)$ and when $x$ lies on $\pi(o,y)$. Therefore, 
$$
G(x,y|\la) = F(x,y|\la)/\la = q^{(n)}(x,y)/\la^{n+1}\,,\quad \text{where }\; n = d(x,y)\,. 
$$
which is non-zero if and only if  and only if $y \in T_x\,$. We can define the $\la$-Martin kernel
as in \eqref{eq:Mkernel}, and \eqref{eq:qn} yields
\begin{equation}\label{eq_MartQ}
K(x,\xi|\la) = K_Q(x,\xi|\la) = 
            \begin{cases} \dfrac{\la^{|x|}}{q^{(|x|)}(o,x)}\,,&\text{if }\; x \in \pi(o,\xi)\,,\\[3pt]
                          0\,,&\text{otherwise.}
            \end{cases}
\end{equation}
It is straightforward that the locally constant function $x \mapsto K_Q(x,\xi|\la)$ is $\la$-harmonic.
In order to make a notational difference from the ``bidirectional'' case, we write 
$\sigma$ instead of $\nu$ for a finitely additive complex distribution on $\partial T$ as in 
\eqref{eq:measure}. Then 
\begin{equation}\label{eq:poissonQ}
h(x) = \int_{\partial T} K_Q(x,\xi|\la)\, d\sigma(\xi) = \frac{\la^{|x|}}{q^{(|x|)}(o,x)}
\, \sigma(\partial T_x)\,.
\end{equation}
is $\la$-harmonic for $Q$. Conversely, if $h$ is $\la$-harmonic, the $\sigma$ can be recovered from
$h$ by \eqref{eq:poissonQ} and a distribution in our sense. We get the following.

\begin{lem}\label{lem:forward}
For $\la \ne 0$, equation \eqref{eq:poissonQ} provides a one-to-one correspondence 
between $\la$-harmonic functions for $Q$ and complex distributions on $\partial T$. 
\end{lem}

As in \S \ref{sec:poly}, $\la$-polyharmonic functions of order $n$ for $Q$ 
are those which are annihilated by
$(\la\,I - Q)^n$. We can proceed exactly as in \S \ref{sec:poly}; the analogue of
Proposition \ref{pro:derive} remains valid, and we have for the $r^{\textrm{th}}$ derivative
with respect to $\la$
\begin{equation}\label{eq:KQ}
K_Q^{(r)}(x,\xi|\la) = |x|(|x|-1) \cdots (|x|-r+1)\, K_Q(x,\xi|\la)
\end{equation}
Now the following is obtained in exactly the same way as Theorem \ref{thm:poly}.

\begin{pro}\label{pro:polyQ} For $\la \in \C \setminus \{0\}$, every $\la$-polyharmonic function 
$f$ of order $n$ for $Q$ has an integral representation
$$
f(x) = \sum_{r=0}^{n-1} \int_{\partial T} K_Q^{(r)}(x,\xi|\la)\,d\sigma_r(\xi)\,,
$$
where the collection of distributions $(\sigma_0\,,\dots, \sigma_{n-1})$ in the sense
of \eqref{eq:measure} is uniquely determined by $f$.  
Conversely, every function which has an integral representation as above
is $\la$-polyharmonic of order $n$ for $Q$.
\end{pro}

For real $t$, expand the polynomial 
$$
t(t-1)\cdots(t-r+1) = \sum_{k=1}^r a_{k,r}\, t^k   
$$
By \eqref{eq:KQ}, we can re-order the integral terms  of 
Proposition \ref{pro:polyQ}:
$$
\sum_{r=0}^{n-1}  K_Q^{(r)}(x,\xi|\la)\,d\sigma_r(\xi)
=  K_Q(x,\xi|\la)\,d\sigma_0(\xi) + 
\sum_{k=1}^{n-1} \sum_{r=k}^{n-1} |x|^k \,a_{k,r}\,  K_Q(x,\xi|\la)\,d\sigma_r(\xi)\,.
$$
Setting $\bar \sigma_0=\sigma_0$ and $\bar \sigma_k = \sum_{r=k}^{n-1}a_{k,r}\,\sigma_r$, 
we get the modified integral  representation
$$
f(x) = \sum_{k=0}^{n-1} |x|^k\, \int_{\partial T} K(x,\xi|\la)\, d\bar\sigma_k(\xi)\,,
$$
and the distributions $\bar\sigma_0\,,\dots, \bar\sigma_{n-1}$ are uniquely determined by $f$,
because the upper diagonal matrix $\bigl( a_{k,r}\bigr)_{1 \le j \le r \le n}$
is invertible. 

\begin{cor}\label{cor:polyQ}
For the forward transition operator $Q$ on $T$ and $\la \in \C \setminus \{0\}$
let $f$ be a $\la$-polyharmonic function of order $n$. Then $f$ has a representation
$$
f(x) = \sum_{k=0}^{n-1} |x|^k\, h_k(x)\,
$$
where the collection of $\la$-harmonic functions $(h_0\,,\dots, h_{n-1})$ for the
forward operator $Q$ is uniquely determined by $f$.  
\end{cor}

This is our analogue of Corollary \ref{cor:poly} 
for arbitrary ``forward only'' transition operators on trees. 
It generalises \cite[Thm. 5.1]{CCGS} to all forward transition operators on countable
trees and and all eigenvalues different from $0$.

\end{document}